\documentclass[12pt]{amsart}

\title{Levi-Civita connections on quantum spheres}
\date{13 Feb 2022}
\author{Joakim Arnlind, Kwalombota Ilwale and Giovanni Landi}

\usepackage{amsmath}
\usepackage{amsthm}
\usepackage{amsfonts}
\usepackage{dsfont}
\usepackage{enumitem}
\usepackage{color}
\usepackage[margin=25mm]{geometry}

\newcommand{\Yp}{Y_{+}}

\newcommand{\Ym}{Y_{-}}
\newcommand{\Yz}{Y_{z}}
\newcommand{\ot}{\otimes}
\newcommand{\nn}{\nonumber}

\newcommand{\complex}{\mathbb{C}}
\newcommand{\naturals}{\mathbb{N}}
\newcommand{\integers}{\mathbb{Z}}

\newcommand{\bracketb}[1]{\Big[#1\Big]}

\newcommand{\angles}[1]{\left\langle #1 \right\rangle}

\newcommand{\para}[1]{\left(#1\right)}
\newcommand{\paraa}[1]{\big(#1\big)}
\newcommand{\parab}[1]{\Big(#1\Big)}

\newcommand{\spacearound}[1]{\quad#1\quad}
\newcommand{\equivalent}{\spacearound{\Leftrightarrow}}
\renewcommand{\implies}{\spacearound{\Rightarrow}}
\newcommand{\qtext}[1]{\quad\text{#1}\quad}

\newcommand{\qand}{\qtext{and}}

\newtheorem{theorem}{Theorem}[section]

\newtheorem{lemma}[theorem]{Lemma}
\newtheorem{proposition}[theorem]{Proposition}

\theoremstyle{definition}
\newtheorem{definition}[theorem]{Definition}
\theoremstyle{remark}
\newtheorem{remark}[theorem]{Remark}
\numberwithin{equation}{section}

\renewcommand{\mid}{\mathds{1}}
\newcommand{\Xp}{X_{+}}
\newcommand{\Xpm}{X_{\pm}}
\newcommand{\Xmp}{X_{\mp}}
\newcommand{\Xm}{X_{-}}
\newcommand{\Xz}{X_{z}}
\newcommand{\omegap}{\omega_{+}}
\newcommand{\omegam}{\omega_{-}}
\newcommand{\omegaz}{\omega_{z}}
\newcommand{\Sthreeq}{S^{3}_{q}}
\newcommand{\Stq}{\Sthreeq}

\newcommand{\TSthreeq}{T\Sthreeq}
\newcommand{\TStq}{\TSthreeq}
\newcommand{\Stwoq}{S^{2}_{q}}
\newcommand{\Uqsu}{\mathcal{U}_q(\textrm{su}(2))}

\newcommand{\A}{\mathcal{A}}

\newcommand{\lt}{\triangleright}
\newcommand{\rt}{\triangleleft}
\newcommand{\one}[1]{#1_{(1)}}
\newcommand{\two}[1]{#1_{(2)}}

\newcommand{\nablat}{\widetilde{\nabla}}
\newcommand{\nablap}{\nabla_{+}}
\newcommand{\nablam}{\nabla_{-}}

\newcommand{\nablaz}{\nabla_{z}}
\newcommand{\nzero}{\nabla^0}

\newcommand{\OmegaStq}{\Omega^1(\Sthreeq)}

\renewcommand{\d}{\partial}

\newcommand{\rtr}{\triangleright}
\newcommand{\ltr}{\triangleleft}

\newcommand{\sigmap}{\sigma_{+}}
\newcommand{\sigmam}{\sigma_{-}}
\newcommand{\sigmaz}{\sigma_{z}}

\newcommand{\Bp}{B_+}
\newcommand{\Bm}{B_-}
\newcommand{\Bz}{B_0}
\renewcommand{\L}{\mathcal{L}}
\newcommand{\eh}{\hat{e}}

\newcommand{\thalf}{\tfrac{1}{2}}

\newcommand{\g}{\mathfrak{g}}
\newcommand{\G}[3]{\Gamma_{#1,#2#3}}
\newcommand{\rhot}{\tilde{\rho}}
\newcommand{\Aeq}[1]{\gamma_{+#1}^\ast-q^2\gamma_{#1-} = A_{#1}}
\newcommand{\Beq}[1]{q^2\rhot_{#1-}-q^{-2}\gamma_{z#1}^\ast=B_{#1}}
\newcommand{\Ceq}[1]{q^2\gamma_{#1z}-q^{-2}\rhot_{#1+}=C_{#1}}

\address[Joakim Arnlind]{Dept. of Math.\\
Link\"oping University\\
581 83 Link\"oping\\
Sweden}
\email{joakim.arnlind@liu.se}

\address[Kwalombota Ilwale]{Dept. of Math.\\
Link\"oping University\\
581 83 Link\"oping\\
Sweden}
\email{kwalombota.ilwale@liu.se}

\address[Giovanni Landi]
{Mathematics, Universit\`a di Trieste \\
34127  Trieste, Italy 
\newline \indent and Institute for Geometry and Physics (IGAP) Trieste, Italy 
\newline \indent and INFN, Trieste, Italy
}
\email{landi@units.it}

\subjclass[2000]{}
\keywords{}

\begin{document}

\begin{abstract}
  We introduce $q$-deformed connections
  on the quantum 2-sphere and 3-sphere, satisfying a twisted Leibniz
  rule in analogy with $q$-deformed derivations. We show that such
  connections always exist on projective modules.  Furthermore, a
  condition for metric compatibility is introduced, and an explicit
  formula is given, parametrizing all metric connections on a free
  module. On the quantum 3-sphere, a
  $q$-deformed torsion freeness condition is introduced and we derive
  explicit expressions for the Christoffel symbols of a Levi-Civita
  connection for a general class of metrics. We also give metric connections on a
  class of projective modules over the quantum 2-sphere. 
  Finally, we outline a generalization to any Hopf algebra with a 
  (left) covariant calculus and associated quantum tangent space.

\end{abstract}

\maketitle

\tableofcontents
\parskip = .25 ex

\section{Introduction}

\noindent
In recent years, a lot of progress has been made in understanding 
Riemannian aspects of noncommutative geometry. These are not only
mathematically interesting, but also important in physics where
noncommutative geometry is expected to play a key role, notably in a
theory of quantum gravity. In Riemannian geometry the Levi-Civita
connection and its curvature have a central role, and it turns
out that there are several different ways of approaching these objects
in the noncommutative setting (see e.g.
\cite{cff:gravityncgeometry,dvmmm:onCurvature,m:nc.spin.q-sphere,ac:ncgravitysolutions,bm:starCompatibleConnections,r:leviCivita,aw:curvature.three.sphere,bgm:levi-civita.class.spectral.triples,bgl2020}).

From an algebraic perspective, the set of vector fields and the set of
differential forms are (finitely generated projective) modules over
the algebra of functions, a viewpoint which is also adopted in
noncommutative geometry. However, considering vector fields as
derivations does not immediately carry over to noncommutative
geometry, since the set of derivations of a (noncommutative) algebra
is in general not a module over the algebra but only a module over the
center of the algebra. Therefore, one is led naturally to focus on
differential forms and define a connection on a general module as
taking values in the tensor product of the module with the module of
differential forms. More precisely, let $M$ be a (right) $\A$-module
and let $\Omega^1(\A)$ denote a module of differential forms together
with a differential $d:\A\to\Omega^1(\A)$. A connection on $M$ is a
linear map $\nabla:M\to M\otimes\Omega^1(\A)$ satisfying a version of
Leibniz rule 
\begin{align}\label{eq:diff.form.Leibniz}
  \nabla(mf) = (\nabla m)f + m\otimes df
\end{align}
for $f\in\A$ and $m\in M$. In differential geometry, for a vector
field $X$ one obtains a covariant derivative $\nabla_X:M\to M$, by
pairing differential forms with $X$ (as differential forms are
dual to vector fields). In a noncommutative version of the above,
there is in general no canonical way of obtaining a ``covariant
derivative'' $\nabla_X:M\to M$. In a derivation based approach to
noncommutative geometry (see
e.g. \cite{dv:calculDifferentiel,dvmmm:onCurvature}), one puts
emphasis on the choice of a Lie algebra $\g$ of derivations of the
algebra $\A$. Given a (right) $\A$-module $M$ one defines a connection
as a map $\nabla:\g\times M\to M$, usually writing
$\nabla(\d,m) = \nabla_\d m$ for $\d\in\g$ and $m\in M$, satisfying
\begin{align*}
\nabla_{\d}(mf) = (\nabla_{\d}m)f + m\,\d(f)
\end{align*}
for $f\in\A$ and $m\in M$, in parallel with \eqref{eq:diff.form.Leibniz}.

For quantum groups, it turns out that natural analogues of vector
fields are not quite derivations, but rather maps satisfying a
twisted Leibniz rule. For instance, as we shall see, for the quantum
3-sphere $\Stq$ one defines maps $X_a:\Stq\to\Stq$ satisfying
\begin{align}\label{eq:Xp.def.Leibniz}
  X_a(fg) = X_a(f) \sigma_a(g) + f X_a(g) 
\end{align}
for $f,g\in\Stq$, and $\sigma_a:\Stq\to\Stq$, for $a=1,2,3$, are
algebra morphisms.  In this note we explore the possibility of
introducing a corresponding $q$-affine connection on a (right)
$\Stq$-module $M$. Motivated by \eqref{eq:Xp.def.Leibniz} we
introduce a covariant derivative $\nabla_{X_a}:M\to M$ such that
\begin{align*}
  \nabla_{X_a}(mf) = (\nabla_{X_a}m)\sigma_a(f) + m X_a(f)
\end{align*}
for $f\in\Stq$ and $m\in M$. In the following, we make these ideas
precise and prove that there exist $q$-affine connections on
projective modules. Furthermore, we introduce a condition for metric
compatibility, and in the particular case of a left covariant calculus
over $\Stq$, we investigate a derivation based definition of torsion.
Then we explicitly construct a Levi-Civita connection, that is a
torsion free and metric compatible connection. Moreover, we construct
metric connections on a class of projective modules over the quantum
2-sphere. We mention that the Riemannian geometry of quantum spheres
was studied \cite{bm:starCompatibleConnections} from the point of view
of a bimodule connection on differential forms satisfying
\eqref{eq:diff.form.Leibniz} as well as a right Leibniz rule twisted
by a braiding map.  In a final section we sketch a way to generalise
(some of) the constructions of the present paper to any Hopf algebra
with a (left) covariant differential calculus and corresponding
quantum tangent space of twisted derivations.

The present paper is an alternative and extended version of the paper
\cite{ail:q.deformed.LC} where the left module structure of
differential forms was used to construct $q$-affine connections,
rather than the right module structure considered in the following.

\section{The quantum 3-sphere}

\noindent
In this section we recall a few basic properties of the quantum
3-sphere \cite{w:twisted.su2}. The algebra $\Sthreeq$ is a unital
$\ast$-algebra generated by $a,a^\ast,c,c^\ast$ fulfilling
\begin{alignat*}{3}
&ac = qca &\qquad & c^*a^* =qa^*c^* &\qquad &ac^* =qc^*a \\
&ca^* =qa^*c & &cc^* = c^*c & &a^*a+c^*c =
aa^*+q^{2}cc^* =\mid 
\end{alignat*}
for a real parameter $q$. The identification of $\Sthreeq$ with the
quantum group $SU_q(2)$ is via the Hopf algebra structure given by
\begin{alignat*}{2}
  &\Delta(a) = a\otimes a-qc^\ast\otimes c &\qquad
  &\Delta(c) = c\otimes a+a^\ast\otimes c\\
  &\Delta(a^\ast) = -qc\otimes c^\ast+a^\ast\otimes a^\ast &\qquad
  &\Delta(c^\ast) = a\otimes c^\ast+c^\ast\otimes a^\ast
\end{alignat*}
with antipode and counit
\begin{alignat*}{4}
  &S(a) = a^\ast &\qquad
  &S(c) = -qc &\qquad
  & \epsilon(a)=1 &\qquad
  & \epsilon(c)=0 \\
  &S(a^\ast) = a &\qquad
  &S(c^\ast) = -q^{-1}c^\ast &\qquad
  & \epsilon(a^\ast)=1 &\qquad
  & \epsilon(c^\ast)=0.
\end{alignat*}
Furthermore, the quantum enveloping algebra $\Uqsu$ is the $\ast$-algebra with generators
$E,F,K,K^{-1}$ satisfying
\begin{align*}
  &K^{\pm 1}E = q^{\pm 1}EK^{\pm 1}\qquad
    K^{\pm 1}F = q^{\mp 1}FK^{\pm 1}\qquad
  [E,F] = \frac{K^2-K^{-2}}{q-q^{-1}} .
\end{align*}
The corresponding Hopf algebra structure is given by the coproduct, 
\begin{align*}
  \Delta(E) = E\otimes K + K^{-1}\otimes E \qquad
  \Delta(F) = F\otimes K + K^{-1}\otimes F
  \qquad \Delta(K^{\pm 1}) = K^{\pm 1}\otimes K^{\pm 1}
\end{align*}
together with antipode and counit
\begin{alignat*}{3}
  &S(K) = K^{-1} &\qquad &S(E) = -qE &\qquad &S(F) = -q^{-1}F   \\
  &\epsilon(K) = 1 & &\epsilon(E) = 0 & &\epsilon(F) = 0.
\end{alignat*}
We recall that there is a unique bilinear pairing between $\Uqsu$ and
$\Sthreeq$ given by
\begin{alignat*}{2}
&\angles{K^{\pm 1}, a} = q^{\mp\, 1/2} &\qquad 
&\angles{K^{\pm 1}, a^*} = q^{\mp\, 1/2}\\
&\angles{E,c} = 1 &  &\angles{F,c^*} = -q^{-1}, 
\end{alignat*}
with the remaining pairings being zero. This 
induces a $\Uqsu$-bimodule structure on $\Sthreeq$: 
\begin{align}\label{actions}
  h\lt f = \one{f}\angles{h,\two{f}}\qand
  f\rt h = \angles{h,\one{f}}\two{f}
\end{align}
for $h\in\Uqsu$ and $f\in\Sthreeq$, with Sweedler's notation
$\Delta(f)=\one{f}\otimes\two{f}$ (and implicit sum).  The
$\ast$-structure on $\Uqsu$, denoted here by $\dag$
(to distinguish it from the $*$-structure of the algebra), is given by
$(K^{\pm 1})^\dag = K^{\pm 1}$ and $E^\dag = F$.  The action of
$\Uqsu$ is compatible with the $\ast$-algebra structures in the
following sense
\begin{align}\label{staractions}
  h\triangleright f^\ast = \paraa{S(h)^\dag\triangleright f}^\ast\qquad
  f^\ast\triangleleft h = \paraa{f\triangleleft S(h)^\dag}^\ast.
\end{align}
Let us for convenience list the left and right actions of the
generators:
\begin{alignat*}{2}
  &K^{\pm 1}\triangleright a^{n}  = q^{\mp\frac{n}{2}}\,a^{n} &\qquad 
  &K^{\pm 1}\triangleright c^{n}  = q^{\mp\frac{n}{2}}\,c^{n}\\
  &K^{\pm 1}\triangleright a^{\ast}\,^{n} = q^{\pm\frac{n}{2}}(a^{\ast})^{n}&\qquad
  &K^{\pm 1}\triangleright c^{\ast}\,^{n} = q^{\pm\frac{n}{2}}(c^{\ast})^{n}\\
  &E\triangleright a^{n}  = -q^{(3-n)/2} [n]a^{n-1}c^{\ast} &\qquad
  &E\triangleright c^{n}  = q^{(1-n)/2}[n]c^{n-1}a^\ast\\
  &E\triangleright (a^{\ast})^{n}  = 0 &\qquad
  &E\triangleright (c^{\ast})^{n}  = 0.\\
  &F\triangleright a^{n}  = 0 &\qquad
  &F\triangleright c^{n}  = 0\\
  &F\triangleright (a^{\ast})^{n}  = q^{(1-n)/2}[n]c(a^{\ast})^{ n-1} &\qquad 
  &F\triangleright (c^{\ast})^{n}  = -q^{-(1+n)/2}[n]a(c^{\ast})^{n-1}
\end{alignat*}
and 
\begin{alignat*}{2}
  &a^n\ltr K^{\pm 1} = q^{\mp \frac{n}{2}}a^n &\quad
  &(a^\ast)^n\ltr K^{\pm 1} = q^{\pm \frac{n}{2}}(a^\ast)^n\\
  &c^n\ltr K^{\pm 1} = q^{\pm \frac{n}{2}}c^n &\quad
  &(c^\ast)^n\ltr K^{\pm 1} = q^{\mp \frac{n}{2}}(c^\ast)^n\\
  &a^n\ltr F = q^{\frac{n-1}{2}}[n]ca^{n-1} &\quad
  &(a^\ast)^n\ltr F = 0\\
  &c^n\ltr F = 0 &\quad
  &(c^\ast)^n\ltr F = -q^{\frac{n-3}{2}}[n]a^\ast(c^\ast)^{n-1}\\
  &a^n\ltr E = 0 &\quad
  &(a^\ast)^n\ltr E = -q^{\frac{n-3}{2}}[n]c^\ast (a^\ast)^{n-1}\\
  &c^n\ltr E = q^{\frac{n-1}{2}}[n]c^{n-1}a &\qquad
  &(c^\ast)^n\ltr E = 0
\end{alignat*}
where $[n] = (q^n-q^{-n})/(q-q^{-1})$.

\subsection{The covariant calculus and the quantum tangent space}\label{sec:left.cov.calculus}

It is well known \cite{w:twisted.su2} that there is a left covariant (first order)
differential calculus on $\Sthreeq$, denoted by $\OmegaStq$, generated
as a left $\Stq$-module by
\begin{align*}
  \omega_1=\omegap = a\,dc-qc\,da\qquad
  \omega_2=\omegam = c^\ast da^\ast-qa^\ast dc^\ast\qquad
  \omega_3=\omegaz = a^\ast da + c^\ast dc .
\end{align*}
In fact, $\OmegaStq$ is a free left module with a
basis given by $\{\omegap,\omegam,\omegaz\}$. Moreover, $\OmegaStq$ is
a bimodule with respect to the relations
\begin{alignat*}{4}
  &\omega_{z}a = q^{-2}a\omega_{z} &\qquad
  &\omega_{z}a^* = q^{2}a^*\omega_{z} &\qquad 
  &\omega_{z}c = q^{-2}c\omega_{z} &\qquad
  &\omega_{z}c^* = q^{2}c^*\omega_{z} \\
  &\omega_{\pm} a = q^{-1}a\omega_{\pm} &\qquad
  &\omega_{\pm}a^* = q a^*\omega_{\pm}&\qquad
  &\omega_{\pm}c = q^{-1}c\omega_{\pm} &\qquad 
  &\omega_{\pm}c^* = qc^*\omega_{\pm},
\end{alignat*}
and, furthermore, $\OmegaStq$ is a $\ast$-bimodule with
\begin{align*}
  &\omegap^\dag= -\omegam\qquad\omegaz^\dag=-\omegaz
\end{align*}
satisfying $(f\omega g)^\dagger=g^\ast\omega^\dag f^\ast$ for
$f,g\in\Sthreeq$ and $\omega\in\OmegaStq$.

The differential $d:\Sthreeq\to\OmegaStq$ is computed using a dual basis $\{\Xp,\Xm,\Xz\}$ of twisted derivations (the corresponding quantum tangent space \cite[\S 14.1.2]{KS97}), 
\begin{align}\label{df3}
  df = (\Xp\rtr f)\omegap + (\Xm\rtr f)\omegam + (\Xz\rtr f)\omegaz , \quad f\in\Sthreeq , 
\end{align}
with explicitly,  
\begin{align*}
\Xp = \sqrt{q}EK\qquad
\Xm  = \frac{1}{\sqrt{q}}FK\qquad
\Xz = \frac{1-K^4}{1-q^{-2}} .
\end{align*}
Their twisted derivation properties are easily worked out. For $f,g\in\Sthreeq$ one has,
\begin{align*}
  X_a \rtr fg = f(X_a \rtr g) + (X_a\rtr f)(\sigma_a \rtr g) ,
 \end{align*}
$a=\pm, z$ (and similarly for the right action), with
\begin{align*}
\sigmap=\sigmam=K^2\quad\text{and}\quad \sigmaz=K^4. 
\end{align*} 
Furthermore, these maps satisfy the following $q$-deformed commutation relations
\begin{align}
  &\Xm\Xp - q^2\Xp\Xm = \Xz\label {eq:Xmp.com} \\
  q^2&\Xz\Xm-q^{-2}\Xm\Xz=(1+q^2) \Xm\label{eq:Xzm.com}\\
  q^2&\Xp\Xz-q^{-2}\Xz\Xp=(1+q^2)\Xp.\label{eq:Xzp.com}
\end{align}
As for the $*$-structures, one checks that $\Xpm^\dagger = \Xmp$ and
$K^\dagger=K$. From this, using \eqref{staractions} one computes, for
$f \in \Stq$, that
\begin{equation}\label{actonstar}
  \begin{split}
    \Xpm \lt f^* &= -(K^{-2} \Xmp \lt f)^* = - K^{2} \lt (\Xmp \lt f)^*  \\
    \Xz\lt f^* &= -(K^{-4} \Xz \lt f)^* = - K^{4} \lt (\Xz \lt f)^* \, .    
  \end{split}
\end{equation}

\section{$q$-affine connections}\label{sec:q.affine.connections}

\noindent
In differential geometry, a connection extends the action of
derivatives to vector fields, and for $\Stq$ a natural set of
($q$-deformed) derivations is given by
$\{X_a\}_{a=1}^3=\{\Xp,\Xm,\Xz\}$. In this section, we will introduce
a framework extending the action of $X_a$ to a connection on
$\Stq$-modules. Let us first define the set of $q$-deformed derivations we shall be
interested in.

\begin{definition}
  The quantum tangent space of $\Stq$ is defined as
  \begin{align*}
    \TStq = \complex\angles{\Xp,\Xm,\Xz},
  \end{align*}
  that is the complex vector space generated by $X_a$ for $a=\pm, z$.
\end{definition}

\noindent
Considering $\TStq$ to be the analogue of a (complexified) tangent space of $\Stq$,
we would like to introduce a covariant derivative $\nabla_X$ on a (right)
$\Stq$-module $M$, for $X\in\TStq$. Since the basis elements of
$\TStq$ act as $q$-deformed derivations, the connection should obey an
analogous $q$-deformed Leibniz rule.  The motivating example is when
$M=\Stq$ and the action of $\TStq$ is simply $\nabla_{X}f = X\lt(f) = X(f)$ for
$X\in\TStq$ and $f\in\Stq$. (To lighten notation, in the following we shall drop the symbol $\lt$ for the left action when there is no risk of ambiguities.)

\noindent
In fact, let us be slightly more general
and consider the action on a free module of rank $n$. Thus, we let $M$
be a free right $\Stq$-module with basis $\{e_i\}_{i=1}^n$, and write
an arbitrary element $m\in M$ as $m=e_im^i$ for $m^i\in\Stq$,
implicitly assuming a summation over $i$ from $1$ to $n$.

Let us define $\nzero:\TStq\times M\to M$ by setting
\begin{align}\label{eq:nzero.def}
  \nzero_{X_a}(m) = e_iX_a(m^i)
\end{align}
for $m=e_im^i\in M$ (and extending it linearly to all of $\TStq$).  Now, it is easy to check that
\begin{align*}
  &\nzero_{X_a}(mf) = (\nzero_{X_a}m)\sigma_a(f) + X_a(f)m
\end{align*}
for $f\in\Stq$ and $m\in M$.  Let us generalize these concepts to
arbitrary right $\Stq$-modules.

\begin{definition}\label{def:q.affine.connection}
  Let $M$ be a right $\Stq$-module. A \emph{right $q$-affine connection on
    $M$} is a map $\nabla:\TStq\times M\to M $ such that
  \begin{enumerate}
  \item $\nabla_{X}(\lambda_1m_1+\lambda_2m_2) = \lambda_1\nabla_Xm_1 + \lambda_2\nabla_Xm_2$,\label{q.affine.lin.module}
  \item $\nabla_{\lambda_1 X+\lambda_2Y}m = \lambda_1\nabla_Xm + \lambda_2\nabla_Ym$,\label{q.affine.lin.tangent}
  \item $\nabla_{X_a}(mf) = (\nabla_{X_a}m)\sigma_a(f)+mX_a(f)$, \quad ($a=\pm, z$),\label{q.affine.Xa.leibniz}
  \end{enumerate}
for $m,m_1,m_2\in M$, $f\in\Stq$, $X\in\TStq$ and
  $\lambda_1,\lambda_2\in\complex$.
\end{definition}

\begin{definition}
  A \emph{hermitian form on a right $\Stq$-module $M$} is
  a map $h:M\times M\to \Stq$ such that
  \begin{align*}
    &h(m_1,m_2f)=h(m_1,m_2)f\qquad  h(m_1,m_2)^\ast = h(m_2,m_1) ,\\
    &h(m_1+m_2,m_3) = h(m_1,m_3)+h(m_2,m_3)
  \end{align*} 
  for $f\in\Stq$ and $m_1,m_2,m_3\in M$. Moreover, $h$ is said to be 
  \emph{invertible} if the induced map $\hat{h}:M\to M^\ast$, defined
  by $\hat{h}(m_1)(m_2)=h(m_1,m_2)$, is bijective.
\end{definition}

\noindent
On a free module with basis $\{e_i\}_{i=1}^n$, a hermitian form is
given by $h_{ij}=h_{ji}^\ast\in \Stq$ by setting
\begin{align*}
  h(m_1,m_2) = (m_1^i)^\ast h_{ij}m_2^j
\end{align*}
for $m_1=e_im_1^i\in(\Stq)^n$ and $m_2=e_im_2^i\in(\Stq)^n$. Moreover,
if $h$ is invertible, then there exist $h^{ij}\in\Stq$ such that
$h^{ij}h_{jk}=\delta^i_k\mid$. In case the module is projective (but
not necessarily free) and generated by $\{e_i\}_{i=1}^n$, one can find
$h^{ij}\in \Stq$ such that $e_ih^{ij}h_{jk}=e_k$ if the hermitian form
is invertible (see e.g. \cite{a:levi.civita.minimal.surfaces}).

Next, we will introduce a notion of compatibility between a $q$-affine connection and a hermitian form.
To motivate Definition~\ref{def:metric.compatibility} below, let us study the case of free modules.
For the $q$-affine connection $\nzero$ in
\eqref{eq:nzero.def}, one finds that
\begin{align*}
  \Xp\paraa{
  h(m_1,m_2)}
  &= \Xp\paraa{(m_1^i)^\ast h_{ij}m_2^j} \\
  &= (m_1^i)^\ast \Xp(h_{ij}m_2^j) + \Xp\paraa{(m_1^i)^\ast}K^2(h_{ij}m_2^j)\\
  &= (m_1^i)^\ast h_{ij}\Xp(m_2^j) + (m_1^i)^\ast \Xp(h_{ij})K^2(m_2^j)+\Xp\paraa{(m_1^i)^\ast}K^2(h_{ij}m_2^j).
\end{align*}
For the connection $\nabla^0$, a natural requirement for the
compatibility with $h$ is to demand that $\Xp(h_{ij})=0$. 
Then,  from \eqref{actonstar} $\Xp(f^\ast)=-(K^{-2}\Xm(f))^\ast = -K^{2}(\Xm(f))^\ast$, and one has,
\begin{align*}
  \Xp\paraa{h(m_1,m_2)}
  &= (m_1^i)^\ast h_{ij}\Xp(m_2^j) +\Xp\paraa{(m_1^i)^\ast}K^2(h_{ij}m_2^j)\\
  &=(m_1^i)^\ast h_{ij}\Xp(m_2^j) - \paraa{K^{-2}\Xm(m_1)}^\ast K^2(h_{ij}m_2^j)\\
  &= (m_1^i)^\ast h_{ij}\Xp(m_2^j) - K^2\paraa{\Xm(m_1)^\ast} K^2(h_{ij}m_2^j)\\
  &= (m_1^i)^\ast h_{ij}\Xp(m_2^j) - K^2\paraa{\Xm(m_1)^\ast h_{ij}m_2^j}\\
  &= h\paraa{m_1,\nzero_{\Xp}m_2}-K^2\paraa{h(\nzero_{\Xm}m_1,m_2)} .
\end{align*}
Corresponding formulas are easily worked out for
$\nzero_{\Xm},\nzero_{\Xz}$, and we shall take this as a motivation
for the following definition.
\begin{definition}\label{def:metric.compatibility}
  A $q$-affine connection $\nabla$ on a right $\Stq$-module $M$ is
  compatible with the hermitian form $h:M\times M\to\Stq$ if
  \begin{align}
    &\Xp\paraa{h(m_1,m_2)}
      = h\paraa{m_1,\nabla_{\Xp}m_2}-K^2\paraa{h(\nabla_{\Xm}m_1,m_2)}\label{eq:metric.comp.p}\\
    &\Xm\paraa{h(m_1,m_2)}
      = h\paraa{m_1,\nabla_{\Xm}m_2}-K^2\paraa{h(\nabla_{\Xp}m_1,m_2)}\label{eq:metric.comp.m}\\
    &\Xz\paraa{h(m_1,m_2)}
      =h\paraa{m_1,\nabla_{\Xz}m_2}-K^4\paraa{h(\nabla_{\Xz}m_1,m_2)},\label{eq:metric.comp.z}
  \end{align}
  for $m_1,m_2\in M$.
\end{definition}

\noindent
Note that \eqref{eq:metric.comp.p} and \eqref{eq:metric.comp.m} are equivalent since
\begin{align*}
  \parab{
  \Xp&\paraa{h(m_2,m_1)}
      -h(m_2,\nabla_{\Xp}m_1)+K^2\paraa{h(\nabla_{\Xm}m_2,m_1)}
  }^\ast\\
     &=-K^{-2}\parab{
       \Xm\paraa{h(m_1,m_2)}+K^2\paraa{h(\nabla_{\Xp}m_1,m_2)}-h(m_1,\nabla_{\Xm}m_2)
       }.
\end{align*}
In the case of a $q$-affine connection on a free module, one can
derive a convenient parametrization of all connections that are
compatible with a given hermitian form. To this end, let us introduce
some notation. Let $(\Stq)^n$ be a free right $\Stq$-module with basis
$\{e_i\}_{i=1}^n$.  A $q$-affine connection $\nabla$ on $(\Stq)^n$ can
be determined by specifying the Christoffel symbols
\begin{align*}
  \nabla_{X_a}e_i = e_j\Gamma_{ai}^j,
\end{align*}
with $\Gamma_{ai}^j\in\Stq$ for $a=\pm, z$ and $i,j=1,\ldots,n$, and
setting
\begin{align*}
  &\nabla_{X_a}(e_im^i) = (\nabla_{X_a}e_i)\sigma_a(m^i) + e_iX_a(m^i)
    = e_j\paraa{\Gamma^j_{ai}\sigma_a(m^i) + X_a(m^j)}.
\end{align*}
The next result gives the form of the Christoffel symbols for a
$q$-affine connection compatible with an invertible hermitian form on
a free module.

\begin{proposition}\label{prop:nabla.comp.metric}
  Let $(\Stq)^n$ be a free right $\Stq$-module with a basis
  $\{e_i\}_{i=1}^n$ and let $\nabla$ be a $q$-affine connection on
  $(\Stq)^n$ given by the Christoffel symbols
  $\nabla_ae_i = e_j\Gamma_{ai}^j$. Furthermore, assume that $h$ is an
  invertible hermitian form on $(\Stq)^n$ and set
  $h_{ij}=h(e_i,e_j)$. Then $\nabla$ is compatible with $h$ if and
  only if there exist $\gamma_{ij},\rho_{ij}\in\Stq$ such that $\rho_{ij}^\ast=\rho_{ji}$ and
  \begin{align}
    &\Gamma^i_{+j} = h^{ik}\paraa{\thalf\Xp(h_{kj})+K(\gamma_{kj})}\label{eq:christoffel.metric.p}\\
    &\Gamma^i_{-j} = h^{ik}\paraa{\thalf\Xm(h_{kj})+K(\gamma_{jk}^\ast)}\label{eq:christoffel.metric.m}\\
    &\Gamma^i_{zj} = h^{ik}\paraa{\thalf\Xz(h_{kj})+K^2(\rho_{kj})}\label{eq:christoffel.metric.z}.
  \end{align}
\end{proposition}

\begin{proof}
  From the compatibility condition \eqref{eq:metric.comp.p} one obtains
  \begin{align*}
    \Xp(h_{ij}) &= h_{ik}\Gamma^k_{+j} - K^2\paraa{(\Gamma^k_{-i})^\ast h_{kj}}
    =h_{ik}\Gamma^k_{+j} - K^2\paraa{(h_{jk}\Gamma^k_{-i})^\ast}
    = \Gamma_{+,ij} - K^{2}\paraa{\Gamma_{-,ji}^\ast},
  \end{align*}
  with $\Gamma_{a,ij}=h_{ik}\Gamma^k_{aj}$, which can be solved by
  \begin{align*}
    \Gamma_{+,ij} = \Xp(h_{ij})+K^2\paraa{\Gamma_{-,ji}^\ast}.
  \end{align*}
  Writing 
  \begin{align*}
    \Gamma_{-,ij} = \thalf\Xm(h_{ij}) + K(\gamma_{ji}^\ast)\implies
    \Gamma_{-,ji}^\ast = -\thalf K^{-2}\Xp(h_{ij}) + K^{-1}(\gamma_{ij}),
  \end{align*}
  for $\gamma_{ij}\in\Stq$, gives
  \begin{align*}
    \Gamma_{+,ij} = \thalf\Xp(h_{ij}) + K(\gamma_{ij})
  \end{align*}
  providing the general solution to \eqref{eq:metric.comp.p} and
  \eqref{eq:metric.comp.m}. Similarly, the Ansatz
  \begin{align*}
    \Gamma_{z,ij} = \thalf\Xz(h_{ij}) + K^2(\rho_{ij})
  \end{align*}
  solves \eqref{eq:metric.comp.z} if and only if
  $\rho_{ij}^\ast=\rho_{ji}$. Using
  $\Gamma^i_{aj} = h^{ik}\Gamma_{a,kj}$ one arrives at
  \eqref{eq:christoffel.metric.p}--~\eqref{eq:christoffel.metric.z}.
\end{proof}

\subsection{$q$-affine connections on projective modules}

\noindent
As expected, one can construct $q$-affine connections on
projective modules. More precisely, one proves the following result.

\begin{proposition}\label{prop:q.affine.projective}
  Let $M$ be a right $\Stq$-module and let $\nabla$ be a $q$-affine
  connection on $M$. Given a projection on $M$, i.e. an endomorphism
  $p:M\to M$ such that $p^2=p$, then $p\circ\nabla$ is a $q$-affine
  connection on the right $\Stq$-module $p(M)$.
\end{proposition}

\begin{proof}
  Since $\nabla$ is a $q$-affine connection and $p$ is an endomorphism, it is immediate that
  $\nablat=p\circ\nabla$ satisfies properties
  (\ref{q.affine.lin.module}) and (\ref{q.affine.lin.tangent}) in
  Definition~\ref{def:q.affine.connection}. Moreover, for $m\in p(M)$
 \begin{align*}
\nablat_{X_a}(m f)
    &=p\paraa{\nabla_{X_a}(m f)} = p\paraa{(\nabla_{X_a}m) \, \sigma_a(f) } + 
    p \paraa{ m X_a(f) }\\
    &=(\nablat_{X_a}m) \, \sigma_a(f) + m X_a(f),
  \end{align*}
since $p(m)=m$ when $m\in p(M)$. We conclude that $\nablat$ is a $q$-affine connection on
  $p(M)$.
\end{proof}

\noindent
Since we have shown in the previous section that 
$q$-affine connections exist on free modules,
Proposition~\ref{prop:q.affine.projective} implies that every projective
$\Stq$-module can be equipped with a $q$-affine connection.
Moreover, let $\nabla$ and $\nablat$
be $q$-affine connections on a $\Stq$-module $M$ and define
\begin{align*}
  \alpha(X,m) = \nabla_Xm-\nablat_Xm.
\end{align*}
Then $\alpha:\TStq\times M\to M$ satisfies
\begin{align}
  &\alpha(\lambda X+\mu Y,m_1) = \lambda\alpha(X,m_1)+\mu\alpha(Y,m_1)\label{eq:alpha.prop.1}\\
  & \alpha(X, m_1 f + m_2 g) = \alpha(X,m_1) f + \alpha(X,m_2) g \label{eq:alpha.prop.2}
\end{align}
for $m_1,m_2\in M$, $X,Y\in\TStq$, $f,g\in\Stq$ and
$\lambda,\mu\in\complex$.  Conversely, every $q$-affine connection on
a projective module $M$ can be written as
\begin{align*}
  \nabla_Xm = p(\nzero_X m) + \alpha(X,m).
\end{align*}
where $\nzero$ is the connection defined in \eqref{eq:nzero.def} and
$\alpha:\TStq\times M\to M$ is an arbitrary map satisfying
\eqref{eq:alpha.prop.1} and \eqref{eq:alpha.prop.2}.
Next, let us show that a connection on a projective module is
  compatible with the restricted metric if the projection is
  orthogonal.

\begin{proposition}
  Let $\nabla$ be a $q$-affine connection on the $\Stq$-module
  $M$ and assume furthermore that $\nabla$ is compatible with a
  hermitian form $h$ on $M$. If $p:M\to M$ is an
  orthogonal projection, i.e. $p$ is a projection such that, for all $m_1,m_2\in M$,
  \begin{align*}
    h\paraa{p(m_1),m_2} = h\paraa{m_1,p(m_2)}
  \end{align*}
 then $\nablat=p\circ\nabla$ is a $q$-affine
  connection on $p(M)$ that is compatible with $h$ restricted to $p(M)$.
\end{proposition}

\begin{proof}
  First of all, it follows from
  Proposition~\ref{prop:q.affine.projective} that
  $\nablat=p\circ\nabla$ is a $q$-affine connection on $p(M)$. Since
  $p$ is an orthogonal projection, one finds that for $m_1,m_2\in p(M)$
  \begin{align*}
    h\paraa{m_1,\nablat_{\Xp}m_2}&-K^2\paraa{h(\nablat_{\Xm}m_1,m_2)}
    = h\paraa{m_1,p(\nabla_{\Xp}m_2)}-K^2\paraa{h(p(\nabla_{\Xm}m_1),m_2)}\\
    &=h\paraa{p(m_1),\nabla_{\Xp}m_2}-K^2\paraa{h(\nabla_{\Xm}m_1,p(m_2))}\\
    &=h\paraa{m_1,\nabla_{\Xp}m_2}-K^2\paraa{h(\nabla_{\Xm}m_1,m_2)}
      = \Xp\paraa{h(m_1,m_2)}
  \end{align*}
  by using that $\nabla$ is compatible with $h$. A similar computation shows that
  \begin{align*}
    \Xz\paraa{h(m_1,m_2)} = h\paraa{m_1,\nablat_{\Xz}m_2}-K^4\paraa{h(\nablat_{\Xz}m_1,m_2)},
  \end{align*}
  from which we conclude that $\nablat$ is compatible with $h$ restricted to $p(M)$.
\end{proof}

\section{A $q$-affine Levi-Civita connection on $\OmegaStq$} \label{se:4}

\noindent
In this section we shall construct a $q$-affine connection on
$\OmegaStq$, compatible with an invertible hermitian form $h$ and satisfying a
certain torsion freeness condition. The module $\OmegaStq$ is a free
$\Stq$-module of rank 3 with basis $\omegap,\omegam,\omegaz$ which
implies that the results of Proposition~\ref{prop:nabla.comp.metric}
may be used.  Although $\OmegaStq$ has a bimodule structure, we shall
only consider the right module structure of $\OmegaStq$ in what
follows. In the case of a $q$-affine connection on $\OmegaStq$, there
is a natural definition of torsion freeness, suggested by the
relations \eqref{eq:Xmp.com}--\eqref{eq:Xzp.com}.

\begin{definition}
  A $q$-affine connection $\nabla$ on $\OmegaStq$ is \emph{torsion free} if
  \begin{align}
    &\nablam\omegap-q^2\nablap\omegam=\omegaz\label{eq:def.torsion.mp}\\
    q^2&\nablaz\omegam-q^{-2}\nablam\omegaz=(1+q^2)\label{eq:def.torsion.zm}\omegam\\
    q^2&\nablap\omegaz-q^{-2}\nablaz\omegap=(1+q^2)\omegap.\label{eq:def.torsion.pz}
  \end{align}
\end{definition}

\noindent
In the following, we will construct a torsion free $q$-affine
connection on $\OmegaStq$ that is compatible with a hermitian form. We
call a connection satisfying these conditions a $q$-affine Levi-Civita
connection. As it turns out, for such connections to exist, the
hermitian form needs to satisfy a compatibility condition.

\begin{proposition}\label{prop:condition.existence.lc}
  Let $h$ be an invertible hermitian form on the (right) $\Stq$-module
  $\OmegaStq$ and write $h_{ab}=h(\omega_a,\omega_b)$. A $q$-affine
  Levi-Civita connection on $\OmegaStq$ exists if and only if
  \begin{align}
  &\Xz(h_{++}-q^2h_{--})
  =K^2\Xm(h_{z+})-q^2\Xm(h_{-z})                    
    -q^2K^2\Xp(h_{z-})+\Xp(h_{+z}).\label{eq:prop.condition.lc.connection}
  \end{align}
\end{proposition}

\begin{proof}
  Assume that $h$ is an invertible hermitian form on $\OmegaStq$
  and write $h_{ab}=h(\omega_a,\omega_b)$ with inverse
  $h^{ab}$. Furthermore, we write $\nabla_a=\nabla_{X_a}$ and 
  \begin{align*}
    \nabla_{a}\omega_b = \omega_c\Gamma_{ab}^c
  \end{align*}
  for $a,b=\pm, z$. In terms of $\Gamma_{a,bc}=h_{bp}\Gamma^p_{ac}$
  the torsion free equations \eqref{eq:def.torsion.mp}--~\eqref{eq:def.torsion.pz} become
  \begin{align}
    &\G{-}{a}{+}-q^2\G{+}{a}{-}=h_{az}\label{eq:torsion.lo.Gamma.mp}\\
    &q^{2}\G{z}{a}{-}-q^{-2}\G{-}{a}{z}=(1+q^2)h_{a-}\label{eq:torsion.lo.Gamma.zm}\\
    &q^{2}\G{+}{a}{z}-q^{-2}\G{z}{a}{+}=(1+q^2)h_{a+}.\label{eq:torsion.lo.Gamma.pz}
  \end{align}
  Since $\OmegaStq$ is a free (right) module, one can apply the
  results of Proposition~\ref{prop:nabla.comp.metric} to obtain
  \begin{align*}
    &\G{+}{a}{b} = \tfrac{1}{2}\Xp(h_{ab})+K(\gamma_{ab})\\
    &\G{-}{a}{b} = \tfrac{1}{2}\Xm(h_{ab})+K(\gamma_{ba}^\ast)\\
    &\G{z}{a}{b} = \tfrac{1}{2}\Xz(h_{ab})+K(\rho_{ab}) ,    
  \end{align*}
for ``parameters" $(\gamma_{ab}, \rho_{ab}=\rho_{ba}^\ast)$ in $\Stq$,
   giving all $q$-affine connections
  compatible with $h$. 

\noindent  
  Inserting the above expressions into
  \eqref{eq:torsion.lo.Gamma.mp}--~\eqref{eq:torsion.lo.Gamma.pz}
  gives
  \begin{align*}
    &\gamma_{+a}^\ast-q^2\gamma_{a-} = K^{-1}(h_{az})
      -\tfrac{1}{2}K^{-1}\Xm(h_{a+})+\tfrac{1}{2}q^2K^{-1}\Xp(h_{a-})\equiv A_a\\
    &q^2K(\rho_{a-})-q^{-2}\gamma_{za}^\ast=K^{-1}\bracketb{(1+q^2)h_{a-}
      -\tfrac{1}{2}q^2\Xz(h_{a-})+\tfrac{1}{2}q^{-2}\Xm(h_{az})}\equiv B_a\\
    &q^2\gamma_{az}-q^{-2}K(\rho_{a+})
    =K^{-1}\bracketb{(1+q^2)h_{a+}-\tfrac{1}{2}q^2\Xp(h_{az})+\tfrac{1}{2}q^{-2}\Xz(h_{a+})}
    \equiv C_a.
  \end{align*}
  Note that the right hand sides $A_a$, $B_a$ and $C_a$ only depend on
  the metric components $h_{ab}$. 
  
  \noindent
  The above 9 equations can be grouped into three independent sets:  
  
  \noindent
  \underline{Group 1}
  \begin{align}
    &\Aeq{+}\tag{G1.1}\\
    &\Aeq{-}\tag{G1.2}
  \end{align}
  
    \noindent
  \underline{Group 2}
  \begin{align}
    &\Beq{+}\tag{G2.1}\\
    &\Beq{z}\tag{G2.2}\\
    &\Ceq{-}\tag{G2.3}\\
    &\Ceq{z}\tag{G2.4}
  \end{align}
  
    \noindent
  \underline{Group 3}
  \begin{align}
    &\Aeq{z}\tag{G3.1}\label{eq:G31}\\
    &\Beq{-}\tag{G3.2}\label{eq:G32}\\
    &\Ceq{+}\tag{G3.3}\label{eq:G33},
  \end{align}
  where, for notational convenience we denoted $\rhot_{ab}=K(\rho_{ab})$.  
  
  \noindent
  The
  equations in Group 1 can be solved as
  \begin{align}
    &\gamma_{++} = A_+^\ast+q^{2}\gamma_{+-}^\ast\label{eq:sol.gpp}\\
    &\gamma_{--} = q^{-2}\gamma_{+-}^\ast-q^{-2}A_-,\label{eq:sol.gmm}
  \end{align}
  and the equations in Group 2 can be solved as
  \begin{align}
    &\gamma_{z+} = q^4\rhot_{+-}^\ast-q^2B_+^\ast\\
    &\rho_{z-} = q^{-2}K^{-1}(B_z)+q^{-4}K^{-1}(\gamma_{zz}^\ast)\\
    &\gamma_{-z} = q^{-2}C_-+q^{-4}\rhot_{-+}\\
    &\rho_{z+} = q^4K^{-1}(\gamma_{zz})-q^2K^{-1}(C_z).\label{eq:sol.rzp}
  \end{align}
  Note that the condition $\rho_{ab}^\ast=\rho_{ba}$ will not pose a
  problem here, since neither $\rho_{-z}$ nor $\rho_{+z}$ appear in
  any other equation, and may simply be defined as
  $\rho_{-z}=\rho_{z-}^\ast$ and $\rho_{+z}=\rho_{z+}^\ast$.
  
  For the equations in Group 3, the fact that we require
  $\rho_{++}^\ast=\rho_{++}$ and $\rho_{--}=\rho_{--}^\ast$ gives a
  non-trivial condition for solutions to exist. From \eqref{eq:G32}
  and \eqref{eq:G33} one obtains
\begin{align}
  &\gamma_{z-} = q^4K^{-2}(\rhot_{--})-q^2B_-^\ast\label{eq:sol.gzm}\\
  &\gamma_{+z} = q^{-2}C_++q^{-4}\rhot_{++},\label{eq:sol.gpz}
\end{align}
and inserted into \eqref{eq:G31} this gives
\begin{align}
  &q^{-4}\rho_{++}-q^6\rho_{--} =
  K(A_z)-q^4K(B_-^\ast)-q^{-2}K(C_+^\ast)\equivalent\notag\\
  &\rho_{++} = q^{10}\rho_{--}+q^4K(A_z)-q^8K(B_-^\ast)-q^2K(C_+^\ast).\label{eq:rhopp.rhomm}
\end{align}
A necessary (and sufficient) condition for solutions to exist, is that
the right hand side of the above equation is hermitian. From
\begin{align*}
  &A_z = K^{-1}\bracketb{h_{zz}
      -\tfrac{1}{2}\Xm(h_{z+})+\tfrac{1}{2}q^2\Xp(h_{z-})}\\
  &B_- = K^{-1}\bracketb{(1+q^2)h_{--}
      -\tfrac{1}{2}q^2\Xz(h_{--})+\tfrac{1}{2}q^{-2}\Xm(h_{-z})}\\
  &C_+ = K^{-1}\bracketb{(1+q^2)h_{++}-\tfrac{1}{2}q^2\Xp(h_{+z})+\tfrac{1}{2}q^{-2}\Xz(h_{++})}
\end{align*}
one obtains
\begin{align*}
  K(A_z) &= h_{zz}
      -\tfrac{1}{2}\Xm(h_{z+})+\tfrac{1}{2}q^2\Xp(h_{z-})\\ ~& \\
  K(B_-^\ast) &= (1+q^2)K^2(h_{--})
                +\tfrac{1}{2}q^2K^{-2}\Xz(h_{--})-\tfrac{1}{2}q^{-2}\Xp(h_{z-})\\
         &= \tfrac{q^2}{2(1-q^{-2})}\paraa{K^2(h_{--})+K^{-2}(h_{--})}
  -\tfrac{q^{-2}}{1-q^{-2}}K^2(h_{--})-\tfrac{1}{2}q^{-2}\Xp(h_{z-})\\ ~& \\
  K(C_+^\ast) &= (1+q^2)K^2(h_{++})+\tfrac{1}{2}q^2\Xm(h_{z+})-\tfrac{1}{2}q^{-2}K^{-2}\Xz(h_{++})\\
         &= -\tfrac{q^{-2}}{2(1-q^{-2})}\paraa{K^2(h_{++})+K^{-2}(h_{++})}+\tfrac{q^2}{1-q^{-2}}K^2(h_{++})
           +\tfrac{1}{2}q^2\Xm(h_{z+})
\end{align*}
by using that $\Xz=(1-K^4)/(1-q^{-2})$. 

\noindent
Since $\rho_{--}$ and
$h_{zz}$, as well as $K^2(h_{--})+K^{-2}(h_{--})$ and
$K^2(h_{++})+K^{-2}(h_{++})$, are hermitian, the non-hermitian terms
of \eqref{eq:rhopp.rhomm}, which we denote by $S$, become
\begin{align*}
  S &= q^6\Xp(h_{z-})-q^4\Xm(h_{z+})+\tfrac{q^{6}}{1-q^{-2}}K^2(h_{--})
  -\tfrac{q^4}{1-q^{-2}}K^2(h_{++}).
\end{align*}

\noindent
Thus, a necessary and sufficient condition for $\rho_{++}$ to be
hermitian is that
\begin{align*}
  0= S-S^\ast
  &= q^6\Xp(h_{z-})+q^6K^{-2}\Xm(h_{-z})
    -q^4\Xm(h_{z+})-q^4K^{-2}\Xp(h_{+z})\\
  &\qquad
    +\tfrac{q^{6}}{1-q^{-2}}K^2(h_{--})-\tfrac{q^{6}}{1-q^{-2}}K^{-2}(h_{--})
    -\tfrac{q^4}{1-q^{-2}}K^2(h_{++})+\tfrac{q^4}{1-q^{-2}}K^{-2}(h_{++})\\
  &= q^6\Xp(h_{z-})+q^6K^{-2}\Xm(h_{-z})
    -q^4\Xm(h_{z+})-q^4K^{-2}\Xp(h_{+z})\\
  &\qquad
    +K^{-2}\Xz\paraa{q^4h_{++}-q^6h_{--}}
\end{align*}
which is equivalent to \eqref{eq:prop.condition.lc.connection}. Hence,
assuming the above relation to hold true, a solution to the torsion
free equations, which is also compatible with $h$, is given by
\eqref{eq:sol.gpp}--\eqref{eq:rhopp.rhomm}. The free parameters in
this solution are
$\gamma_{+-}, \, \gamma_{-+}, \, \gamma_{zz}, \, \rho_{+-}$ and
$\rho_{--}^\ast=\rho_{--}, \, \rho_{++}^\ast=\rho_{++}$.  
\end{proof}

\noindent
Although the general $q$-affine Levi-Civita connection on $\OmegaStq$
may be written down, the expressions are rather lengthy and not
particularly illuminating. However, let us explicitly write down a
Levi-Civita connection in the particular case of a diagonal metric of
the form
\begin{align*}
  &h_{--} = h, \qquad
  h_{++}=q^2h, \qquad
    h_{zz}=h_z, \qquad h_{ab}=0\text{ if }a\neq b,
\end{align*}
with $h$ and $h_z$  invertible elements of $\Stq$; note that the
above choice clearly satisfies \eqref{eq:prop.condition.lc.connection}
in Proposition~\ref{prop:condition.existence.lc}. Using the solution
given by \eqref{eq:sol.gpp}--\eqref{eq:rhopp.rhomm}
in the proof of
Proposition~\ref{eq:prop.condition.lc.connection} one finds that
\begin{align*}
  &\nabla_{\Xp}\omega_a = \omega_bh^{bc}\paraa{\thalf\Xp(h_{ca})+K(\gamma_{ca})}\\
  &\nabla_{\Xm}\omega_a = \omega_bh^{bc}\paraa{\thalf\Xm(h_{ca})+K(\gamma_{ac}^\ast)}\\
  &\nabla_{\Xz}\omega_a= \omega_bh^{bc}\paraa{\thalf\Xz(h_{ca})+K^2(\rho_{ca})}
\end{align*}
with 
\begin{align*}
  &\gamma_{++} = \tfrac{1}{2}q^2K^{-1}\Xp(h)+q^2\gamma_{+-}^\ast\\
  &\gamma_{--} = -\tfrac{1}{2}K^{-1}\Xp(h)+q^{-2}\gamma_{+-}^\ast\\
  &\gamma_{+z} = (1+q^2)K^{-1}(h)+\tfrac{1}{2}q^{-2}K^{-1}\Xz(h)\\
  &\gamma_{z+} = q^4K^{-1}(\rho_{-+})\\
  &\gamma_{-z} = q^{-4}K(\rho_{-+})\\
  &\gamma_{z-} =q^4K^{-1}(\rho_{--})-q^2(1+q^2)K^{-1}(h)-\tfrac{1}{2}q^4K^{-3}\Xz(h)
\end{align*}
and
\begin{align*}
  &\rho_{z+} = q^4K^{-1}(\gamma_{zz})+q^4K^{-2}\Xp(h_z)\\
  &\rho_{+z} = \rho_{z+}^\ast = q^4K(\gamma_{zz}^\ast)-q^4\Xm(h_z)\\
  &\rho_{z-} = \tfrac{1}{2}q^{-4}K^{-2}\Xm(h_z)+q^{-4}K^{-1}(\gamma_{zz}^\ast)\\
  &\rho_{-z} = \rho_{z-}^\ast = -\tfrac{1}{2}q^{-4}\Xp(h_z)+q^{-4}K(\gamma_{zz})\\
  &\rho_{++} = q^{10}\rho_{--}+q^4h_z
    -\tfrac{1}{2}q^4(1+q^2)(1+q^4)\paraa{K^2(h)+K^{-2}(h)}.
\end{align*}
Furthermore, setting $\gamma_{+-}=\rho_{-+}=\gamma_{zz}= \rho_{--}=0$ one obtains
\begin{align*}
  &\nabla_+\omega_+ = \omega_+h^{-1}\Xp(h)\\
  &\nabla_+\omega_- = \omega_z\frac{h_z^{-1}}{1-q^{-2}}
    \parab{\paraa{1-\tfrac{1}{2}q^4}K^2(h)-\tfrac{1}{2}q^4K^{-2}(h)}\\
  &\nabla_+\omega_z =
    \omega_+q^{-2}h^{-1}\parab{K^2(h_z)+\tfrac{1}{1-q^{-2}}\paraa{(q^2-\tfrac{1}{2}q^6)h-\tfrac{1}{2}q^6K^4(h)}}
    +\omega_z\tfrac{1}{2}h_z^{-1}\Xp(h_z)\\
  &\nabla_-\omega_+ = \omega_z
  +\omega_z\frac{h_z^{-1}}{1-q^{-2}}\parab{\paraa{q^2-\tfrac{1}{2}q^6}K^2(h)-\tfrac{1}{2}q^6K^{-2}(h)}\\
  &\nabla_-\omega_- = \omega_-h^{-1}X_-(h)\\
  &\nabla_-\omega_z =
    \omega_-\frac{h^{-1}}{1-q^{-2}}\parab{(1-\tfrac{1}{2}q^4)h-\tfrac{1}{2}q^4K^4(h)}
    +\omega_z\tfrac{1}{2}h_z^{-1}\Xm(h_z)\\
  &\nabla_z\omega_+ =
    \omega_z\thalf q^4h_z^{-1}\Xp(h_z)+\omega_+q^2h^{-1}K^2(h_z)
    +\omega_+\frac{h^{-1}}{1-q^{-2}}\parab{(1-\tfrac{1}{2}q^8)h-\tfrac{1}{2}q^8K^4(h)}\\
  &\nabla_z\omega_- = \omega_z\tfrac{1}{2}h_z^{-1}q^{-4}\Xm(h_z)
  +\omega_-\frac{h^{-1}}{1-q^{-2}}\paraa{\tfrac{1}{2}h-\tfrac{1}{2}K^4(h)}\\
  &\nabla_z\omega_z = \omega_z\tfrac{1}{2}h_z^{-1}\Xz(h_z)
    -\omega_+\tfrac{1}{2}q^2h^{-1}K^2\Xm(h_z)
    -\omega_-\tfrac{1}{2}q^{-4}h^{-1}K^2\Xp(h_z)  
\end{align*}
giving a $q$-affine Levi-Civita connection on $\OmegaStq$ with respect to the hermitian form $h$.

\section{The quantum 2-sphere}

\noindent
The noncommutative (standard) Podle\'s sphere $S^2_q$
\cite{p:quantum.spheres} can be considered as a subalgebra of $\Stq$
by identifying the generators $\Bz,\Bp,\Bm$ of $\Stwoq$ as
\begin{align*}
  \Bz = cc^\ast\qquad
  \Bp = ca^\ast\qquad
  \Bm = ac^\ast = \Bp^\ast ,
\end{align*}
satisfying then the relations
\begin{alignat*}{2}
\Bm\,\Bz &= q^{2}\, \Bz\,\Bm   &\qquad &\Bp\,\Bz = q^{-2}\, \Bz\,\Bp    \\ 
\Bm\,\Bp &= q^2\, \Bz \,\big( \mid - q^2\, \Bz \big)  &\qquad &\Bp\,\Bm= \Bz \,\big( \mid - \Bz \big) . 
\end{alignat*}
These elements generate the fix-point algebra of the right $U(1)$-action
\begin{align}\label{actu1}
  \alpha_z(a) =az\qquad\alpha_z(a^\ast)=a^\ast\bar{z}\qquad
  \alpha_z(c) =cz\qquad\alpha_z(c^\ast)=c^\ast\bar{z}
\end{align}
for $z\in U(1)$ and $a\in\Stq$, related to the $U(1)$-Hopf-fibration $\Stwoq\hookrightarrow \Stq$. 
Equivalently, the sphere $\Stwoq$ is the invariant subalgebra of 
$\Stq$ for the left action of $K$: $\Stwoq = \{f \in \Stq \, , \, K \rtr f = f \}$.
Then, the left action of the $X_a$ does not preserve the algebra $\Stwoq$ 
(since their left action does not commute with that of $K$):  one readily computes,
\begin{alignat*}{3}
  &\Xp \rtr \Bz = q a^\ast c^\ast &\quad
  &\Xm \rtr \Bz = -q^{-1} ca &\quad
  &\Xz \rtr \Bz = 0\\
  &\Xp \rtr \Bp = q (a^\ast)^2 &\quad
  &\Xm \rtr \Bp = c^2 &\quad
  &\Xz \rtr \Bp = 0\\
  &\Xp \rtr \Bm = q^{2} (c^\ast)^2 &\quad
  &\Xm \rtr \Bm = -q^{-1} (a)^2 &\quad
  &\Xz \rtr \Bm = 0.
\end{alignat*}

\noindent
On the other hand, the right action of $X_a$ does preserve the algebra $\Stwoq$
(since their right action does commute with the left one of $K$). 
Let us denote $Y_a=X_a$ for the right action. Then, it is easy to check that
\begin{alignat*}{3}
  &\Bz\ltr\Yp = q^{-1}\Bm &\quad
  &\Bz\ltr\Ym = -q^{-1}\Bp &\quad
  &\Bz\ltr\Yz = 0\\
  &\Bp\ltr\Yp = q\mid -q(1+q^2)\Bz &
  &\Bp\ltr\Ym = 0 &
  &\Bp\ltr\Yz = -q^2(1+q^2)\Bp\\
  &\Bm\ltr\Yp = 0 &
  &\hspace{-15mm}\Bm\ltr\Ym = -q^{-1}\mid + q^{-1}(1+q^2)\Bz &
  &\Bm\ltr\Yz = (1+q^{-2})\Bm.
\end{alignat*}
Note that when restricted to $\Stwoq$ the $Y_a$ are not independent. 
A long but straightforward computation shows that they are indeed related as
\begin{equation}\label{rel-rvf}
  \begin{split}
    \big( (f \ltr \Yp) \Bp \, q & + (f \ltr \Ym)  \Bm \, q^{-1} \big)(1+q^2) +
    (f \ltr \Yz) \para{ 1-2 \frac{1+q^2}{1+q^4} \Bz}\\
    & = (f \ltr \Yz^2) \, q^{-2} \para{ \frac{1-q^2}{1+q^4} \, ( 2 q^4 + q^2 + 1 ) \Bz - (1-q^6) \Bz^2} \\
    &\quad+ (f \ltr K^4) \, q^{-2} (1+q^2) \paraa{ (q^4 -1) \Bz + (1-q^6) \Bz^2},
  \end{split}
\end{equation}
for $f\in\Stwoq$.  This can be checked on a vector space basis
  for the algebra $\Stwoq$, a basis which can be taken as $X(m) (\Bz)^n$ for
  $m\in \integers$, $n \in\naturals$ with $X(m) = (\Bp)^m$ for
  $m\geq 0$ and $X(m) = (\Bm)^{-m}$ for $m < 0$ (cf. \cite{mnw1991}).

\subsection{A left covariant calculus on $\Stwoq$} 
Since the element $K$ acts (on the left) as the identity on $\Stwoq$, the differential
\eqref{df3} when restricted to $f\in\Stwoq$ becomes
\begin{align}\label{diff-s2l}
d f  =  (\Xm \rtr f) \,\omega_{-} + (\Xp \rtr f) \,\omega_{+} .
\end{align}
 In particular one finds
\begin{align*}
d\Bp &=  q \, (a^\ast)^{2} \, \omegap + c^{2} \, \omegam ,  \\
d\Bm &=  - q^{2}\, (c^\ast)^{2} \, \omegap - q^{-1} \, a^{2} \, \omegam,  \\
d\Bz & =  c^\ast a^\ast\, \omegap -q^{-1} c a  \, \omegam 
\end{align*}
which can be inverted to yield
\begin{align*}
\omegap & = q^{-1} a^2 \, d\Bp -q^2 c^2 \, d\Bm + (1+q^2) ac \, d\Bz  \\
\omegam & = (c^\ast)^{2} \, d\Bp - q (a^\ast)^{2} \, d\Bm - (1+q^2) c^\ast a^\ast \, d\Bz,
\end{align*}
implying that the differential in \eqref{diff-s2l} can be expressed as
\begin{align}\label{diff-s2lbis}
d f & = \paraa{ q^{-1} (\Xp \rtr f) \, a^2 + (\Xm \rtr f) \, (c^\ast)^{2}}d\Bp \nn \\
& \qquad - 
\paraa{ q^2 (\Xp \rtr f) \, c^2 + q (\Xm \rtr f) \, (a^\ast)^{2}} d\Bm \\ 
& \qquad + (1+q^2)\paraa{ (\Xp \rtr f) \, ac - (\Xm \rtr f) \, c^\ast a^\ast}d\Bz .  \nn 
\end{align}
Note that $\Xpm \rtr f \notin \Stwoq$. However, from the commutation relations $K \Xpm = q^{\mp} \Xmp K$ one infer that all coefficients are in $\Stwoq$. For instance: $K \lt \paraa{(\Xp \rtr f) \, a^2} = \paraa{(K \Xp \rtr f) \, K \lt a^2} = 
( q \Xp \rtr f) \, q^{-1} \lt a^2 = (\Xp \rtr f) \, a^2$, and similarly for the other terms. 

\noindent
The form \eqref{diff-s2l} of the differential that uses left invariant vector fields and forms 
can be seen as identifying the cotangent bundle of $S^2$ with the direct sum of the line bundles of `charge' $\pm 2$, that is   
$\Omega^{1}(S^2) \simeq  \L_{-2} \omega_{-} \oplus \L_{+2} \omega_{+}$. This identification can be used also for the quantum 
sphere $\Stwoq$ with the line bundles defined as in \eqref{linbun} below. 

\noindent
One the other hand, from the expression in \eqref{diff-s2lbis}
one writes the differential $d$ on $\Stwoq$ 
in terms of the right acting operators $Y_a$.
\begin{lemma}
For $f\in\Stwoq$, the differential in \eqref{diff-s2l} can be written as
\begin{equation}\label{exd-q}
d f = (f \ltr V_+) \, d\Bp + (f\ltr V_-) \,\, d\Bm + (f \ltr V_0) \, d\Bz  
\end{equation}
where
\begin{align*}
V_+ & = \Yp\paraa{ 1 - q^{-2}(1+q^2) \Bz} q^{-1} - \Yz \,\, \Bm  \frac{q^{-2}(1+q^6)}{1+q^4} 
+ \Yz^2 \,\, \Bm\frac{1-q^2}{(1+q^2)(1+q^4)} 
\\~~\\
V_- & = - \Ym\paraa{ 1 - q^2(1+q^2) \Bz} q + \Yz \,\, \Bp  \frac{q^{-2}(1+q^6)}{1+q^4} 
- \Yz^2 \,\, \Bp \frac{1-q^2}{(1+q^2)(1+q^4)}
\\~~\\
V_0 & = \big( \Yp \,\, \Bp \, q^{-1} -  \Ym \,\, \Bm \, q \big) (1+q^2) + \Yz \,\, \Bz \frac{(1-q^4)(1+q^6)}{1+q^4}   
- \Yz^2 \,\, \Bz \frac{1-q^2}{1+q^4} .
\end{align*}
\end{lemma}
\begin{proof}
By acting on the vector space basis $X(m)(\Bz)^n$ (as introduced previously), one explicitly checks the equality of \eqref{diff-s2l} and \eqref{exd-q} via a tedious but straightforward computation.
\end{proof}

\begin{remark}
  When $q=1$ the derivative \eqref{exd-q} reduces to
  \begin{equation}\label{classdf}
    \begin{split}
      d f &= 2 \paraa{ ( f \ltr \Yp) \,\, \Bp  - (f \ltr \Ym) \,\, \Bm} d\Bz \\
      & \quad + \paraa{ ( f \ltr \Yp)  \,\, (1-2 \Bz) - ( f \ltr \Yz)  \,\, \Bm } d\Bp \\
      & \quad + \paraa{ - ( f \ltr \Ym) \,\, (1-2 \Bz) +  (f \ltr \Yz) \,\, \Bp} d\Bm \, .       
    \end{split}
  \end{equation}
Classically, the vector field $X_a$ are the left invariant vector fields on $S^3 = SU(2)$ 
with dual left invariant forms $\omega_a$. Thus they do not project to vector fields on the base space $S^2$  
with commuting coordinates $(\Bp, \Bm, \Bz)$ and relation $\Bp \Bm = \Bz (1-\Bz)$: 
$X_a \rtr f $ is not a function on $S^2$ even when $f$ is. On the other hand, the vector fields
$Y_a$ are the right invariant vector fields on $SU(2)$ and thus they project to vector fields on $S^2$, where they are not independent any longer and are related by  
\begin{align*} 
2 ( \Bp \Yp + \Bm \Ym ) + (1-2 \Bz) \Yz = 0 ,
\end{align*}
which is just the relation to which \eqref{rel-rvf} reduces when $q=1$.

By changing coordinates $B_0 = \tfrac{1}{2}(1-x)$ so that the radius condition for $S^2$ is written as $r^2 = 4 \Bp \Bm + x^2$, 
the exterior derivative operator in \eqref{classdf} becomes
\begin{align*} 
d f = \partial_x f \, d x + \partial_+ f \, d \Bp + \partial_- f \, d \Bm 
- (\Delta f) \, ( x\, dx + 2 \Bm \, d\Bp + 2 \Bp \, d\Bm ) 
\end{align*}
where $\Delta = x \, \partial_x + \Bp \, \partial_+ + \Bm \, \partial_-$ is the Euler (dilatation) vector field. One then computes 
$d r^2 = 2 ( 1 - r^2 ) ( x\, dx + 2 \Bm \, d\Bp + 2 \Bp \, d \Bm )$, which vanishes when restricting to $S^2$: $r^2-1=0$. 
\end{remark}

\subsection{Connections on projective modules over $\Stwoq$}
In this section, we construct $q$-affine connections on a class of
projective modules over $\Stwoq$.  The definition of $q$-affine
connections applies equally well to the subalgebra $\Stwoq$, and since
$a\ltr Y\in\Stwoq$ for $a\in\Stwoq$ and $Y\in\Uqsu$, we will in the
following consider the right action of $\Uqsu$ on $\Stwoq$.

The quantum Peter--Weyl theorem for $\Stq$ results into an explicit
(vector space) decomposition of the algebra $\Stq$, that is
$\Stq = \oplus_{n\in\integers}\L_n$, with
\begin{align} \label{linbun}
  \L_n = \{f\in\Stq:\alpha_z(f) = \bar{z}^nf\},
\end{align}
for the $U(1)$ action $\alpha_z$ in \eqref{actu1}. Equivalently,
$\L_n = \{f \in \Stq \, , \, K \rtr f = q^{-\frac{n}{2}}f \}$.  It
follows that $\L_0 = \Stwoq$, as well as $\L_n\L_m\subseteq\L_{n+m}$.
Clearly, the right action of $\Uqsu$ leaves each $\L_n$ invariant.
Moreover, it is easy to see that $\L_n$ is a $\Stwoq$-bimodule.
For $f, g \in \Stwoq$ and $\psi_n\in\L_n$,
\begin{align*}
 \alpha_z(f \psi_n g) = \alpha_z(f) \alpha_z(\psi_n) \alpha_z(g) = \bar{z}^n (f \psi_n g),
\end{align*}
which implies that $\L_n$ is a $\Stwoq$-bimodule.

\noindent
As a right (or equivalently left) module, each $\L_n$ can be realised as a
finitely generated projective $\Stwoq$-module as we now briefly recall (cf. \cite{bm:line.bundles.quantum.spheres,hm:projective.monopole,l:twisted.sigma.model}).

For $n\geq 0$ and $\mu=0,1,\ldots,n$, let
$(\Psi_n)_\mu,(\Phi_n)_\mu\in\Stq$ be given as
\begin{align*}
  (\Phi_n)_\mu = \sqrt{\alpha_{n\mu}}c^{n-\mu}a^{\mu} \qquad
  \quad
 (\Psi_n)_\mu = \sqrt{\beta_{n\mu}}(c^\ast)^{\mu}(a^\ast)^{n-\mu}
\end{align*}
with
\begin{align*}
  \alpha_{n\mu}=\prod_{k=0}^{n-\mu-1}\frac{1-q^{2(n-k)}}{1-q^{2(k+1)}}\qquad\qquad
  \beta_{n\mu}=q^{2\mu}\prod_{k=0}^{\mu-1}\frac{1-q^{-2(n-k)}}{1-q^{-2(k+1)}}.
\end{align*}
It is straight-forward to check that
\begin{align*}
  \sum_{\mu=0}^{n}(\Phi_n)_{\mu}^\ast(\Phi_n)_{\mu}=
  \sum_{\mu=0}^{n}(\Psi_n)_{\mu}^\ast(\Psi_n)_{\mu}= \mid,
\end{align*}
implying that
\begin{align*}
  &{(p_{n})^{\mu}}_{\nu} = (\Psi_n)_\mu(\Psi_n)^\ast_{\nu}
  =\sqrt{\beta_{n\mu}\beta_{n\nu}}(c^\ast)^{\mu}(a^\ast)^{n-\mu}a^{n-\nu}c^{\nu}\\
  &{(p_{-n})^{\mu}}_{\nu} = (\Phi_n)_\mu(\Phi_n)^\ast_{\nu}
  =\sqrt{\alpha_{n\mu}\alpha_{n\nu}}c^{n-\mu}a^\mu(a^\ast)^\nu (c^\ast)^{n-\nu}
\end{align*}
satisfy $p_n^2=p_n$ and $p_{-n}^2=p_{-n}$. Moreover, it is easy to see that
the entries ${(p_{n})^{\mu}}_{\nu}$ and ${(p_{-n})^{\mu}}_{\nu}\in\Stwoq$, which implies that
one has finitely generated projective $\Stwoq$-modules
\begin{align*}
  M_n =
  \begin{cases}
    p_n(\Stwoq)^{n+1}&\text{if }n\geq 0\\
    p_{-|n|}(\Stwoq)^{|n|+1}&\text{if }n<0 \, .
  \end{cases}
\end{align*}
 These modules $M_n$ are isomorphic as right $\Stwoq$-modules to $\L_n$ for each $n\in\integers$.

Now, let $\{e_\mu\}_{\mu=0}^n$ be a basis of $(\Stwoq)^{n+1}$. Given
an invertible hermitian form $h$ on $(S^2_q)^{n+1}$,
Proposition~\ref{prop:nabla.comp.metric} gives a $q$-affine connection
on $(\Stwoq)^{n+1}$ compatible with $h$ as
\begin{align*}
  &\nablat_{\Xp}e_\mu = e_{\nu}\Gamma^\nu_{+\mu} =
    e_{\nu}h^{\nu\rho}\paraa{\thalf h_{\rho\mu}\ltr\Yp+a_{\rho\mu}\ltr K}\\
  &\nablat_{\Xm}e_\mu = e_{\nu}\Gamma^\nu_{-\mu}
    = e_{\nu}h^{\nu\rho}\paraa{\thalf h_{\rho\mu}\ltr\Ym+a_{\mu\rho}^\ast\ltr K}\\
  &\nablat_{\Xz}e_{\mu} = e_\nu\Gamma^\nu_{z\mu}
    = e_\nu h^{\nu\rho}\paraa{\thalf h_{\rho\mu}\ltr\Yz+b_{\rho\mu}\ltr K^2}.  
\end{align*}
for arbitrary $a_{\mu\nu},b_{\mu\nu}\in\Stwoq$ such that $b_{\mu\nu}^\ast=b_{\nu\mu}$.

\noindent
If $n\geq 0$ then $\eh_\mu=e_{\nu}{(p_n)^\nu_\mu}$ are generators of
$M_n=p_n(\Stwoq)^{n+1}$ and Proposition~\ref{prop:q.affine.projective}
implies that $\nabla=p_n\circ\nablat$ is a $q$-affine connection on
$M_n$ with
\begin{align*}
  \nabla_{\Xp}\eh_{\mu}
  &= p_n\paraa{\nablat_{\Xp}e_\nu{(p_n)^\nu}_{\mu}}
    = p_n(\nablat_{\Xp}e_\nu)\paraa{{(p_n)^\nu}_{\mu}\ltr K^2}+\eh_{\nu}\paraa{{(p_n)^\nu}_{\mu}\ltr\Xp}\\
  &= \eh_{\gamma}h^{\gamma\rho}q^{2(\mu-\nu)}\paraa{\thalf h_{\rho\nu}\ltr\Xp+a_{\rho\nu}\ltr K}{(p_n)^\nu}_{\mu}
    +\eh_{\nu}\paraa{{(p_n)^\nu}_{\mu}\ltr\Xp}\\
  \nabla_{\Xm}\eh_{\mu}
    &= \eh_{\gamma}h^{\gamma\rho}q^{2(\mu-\nu)}\paraa{\thalf h_{\rho\nu}\ltr\Xm+a^\ast_{\nu\rho}\ltr K}{(p_n)^\nu}_{\mu}
    +\eh_{\nu}\paraa{{(p_n)^\nu}_{\mu}\ltr\Xm}\\
  \nabla_{\Xz}\eh_{\mu}
    &= \eh_{\gamma}h^{\gamma\rho}q^{4(\mu-\nu)}\paraa{\thalf h_{\rho\nu}\ltr\Xz+b_{\rho\nu}\ltr K^2}{(p_n)^\nu}_{\mu}
    +\eh_{\nu}\paraa{{(p_n)^\nu}_{\mu}\ltr\Xz}
\end{align*}
using that
${(p_n)^\mu}_{\nu}\ltr K=q^{\nu-\mu}{(p_n)^\mu}_{\nu}$. Moreover, if
$p_n$ is orthogonal with respect to $h$, then $\nabla$ is compatible
with the restriction of $h$ to $M_n$.
A similar construction goes for $n<0$.  

\section{Further comments: sketching a generalization}

\noindent
As final section of comments we sketch a way to generalise (some of) the constructions above for 
any Hopf algebra with a left covariant differential calculus and corresponding quantum tangent space
\cite{w:qts89}. While referring to \cite[14.1]{KS97} for details, we recall that a first order differential calculus 
$(\Gamma, d)$ over the Hopf algebra $(H, \Delta, S, \varepsilon)$ is called left-covariant if there is a linear map $\Delta_\Gamma : \Gamma \to H \otimes \Gamma$ such that, for all $a,b \in H$ it holds that 
$$
\Delta_\Gamma(f \, dg) = \Delta(f) ({\rm id} \ot d)) \Delta(g) .
$$
An element $\rho \in \Gamma$ is called left-invariant if $\Delta_\Gamma(\rho) = 1 \ot \rho$ and we let  
${}_{\rm inv}\Gamma$ denote the vector space of invariant elements. 
There is then a corresponding quantum tangent space $T_\Gamma \subset H^{\circ}$ (the dual Hopf algebra) with a unique bilinear form 
$\angles{\cdot , \cdot} : T_\Gamma \times \Gamma \to \complex$ such that $T_\Gamma$ 
$$
\angles{X, f \, d g} = \varepsilon(f) X (g) ,  
$$
for $g, f \in H$, and $X\in T_\Gamma$. The vector spaces ${}_{\rm inv}\Gamma$ and $T_\Gamma$ 
form a non-degenerate dual pair with respect to this bilinear form. Also, the pairing can be written as 
a left action as in \eqref{actions}, 
\begin{align*}
  X \lt f = \one{f}\angles{h,\two{f}}
\end{align*}
for $X\in T_\Gamma$ and $f\in H$. Furthermore, one has dual bases $\{X_a, \, a = 1, 2, \dots , n \}$ of $T_\Gamma$ and 
$\{\omega_a, \, a = 1, 2, \dots , n\}$ of ${}_{\rm inv}\Gamma$ 
and a family of functionals $\{\sigma^a_b, \, a, b = 1, 2, \dots , n \}$ such that 
\begin{align}
d f & = \sum\nolimits_a (X_a \lt f ) \, \omega_a \notag\\
X_a \lt (f g ) & = f X_a \lt (g ) + X_b \lt (f) \, \sigma^b_a \lt (g) \label{xacts} .
\end{align} 
In $H^{\circ}$ we have 
\begin{align*}
\Delta{\sigma^a_b} & = \sigma^a_c \ot \sigma^c_b \, , \qquad S(X_a) = - X_b S(\sigma^b_a) \, .
\end{align*}
With compatible $*$-structures, using the second relation and \eqref{staractions} one readily computes:
\begin{align}\label{xacts*}
X_a \lt f^* = - \sigma^b_a \lt (X^\dagger_b \lt f)^* .
\end{align}  

By way of illustration let us consider the trivial right module $M=H$ with hermitian form 
$h(m_1, m_2) = m_1^* m_2$. The analogue of the condition \eqref{q.affine.Xa.leibniz} 
in Definition~\ref{def:q.affine.connection} is read from \eqref{xacts} as
\begin{align}\label{nablagen} 
\nabla_{X_a} \lt (m f) & = m \, X_a \lt (f) + \paraa{\nabla_{X_b} \lt (m)} \, \sigma^b_a \lt (f) .  
\end{align}
In turn, the compatibility with the hermitian form reads:
\begin{align}\label{xhcomp}
X_a\paraa{h(m_1,m_2)}
      = h\paraa{m_1,\nabla_{\Xp}m_2} - \sigma^b_a \lt \paraa{h(\nabla_{X^\dagger_b}m_1,m_2)} .
\end{align}
Indeed, using \eqref{xacts} and \eqref{xacts*}, we compute 
\begin{align*}
X_a\paraa{h(m_1,m_2)} & = X_a \lt(m_1^* m_2)  \\ 
&= X_a \lt ( m_1^* m_2 ) = m_1^* X_a \lt (m_2) +  X_b \lt (m_1^*) \, \sigma^b_a \lt (m_2) \\
&= X_a \lt ( m_1^* m_2 ) = m_1^* X_a \lt (m_2)  
- \sigma^c_b \lt (X^\dagger_c \lt m_1)^*\, \sigma^b_a \lt (m_2) \\
 &= X_a \lt ( m_1^* m_2 ) = m_1^* X_a \lt (m_2)  
- \sigma^b_a \lt \paraa{X^\dagger_b \lt m_1)^* \, (m_2)} 
\end{align*}
from which \eqref{xhcomp} follows. 

Equations \eqref{nablagen} and \eqref{xhcomp} can be the starting
point for a theory of affine connections on a quantum group with a
quantum tangent space. 
For a torsion freeness condition one would need 
(twisted) commutation relations among the elements of $T_\Gamma$. 
In general these commutation relations could be involved; in particular they do not need to be quadratic as in the classical case or in the example in \eqref{eq:Xzp.com}--\eqref{eq:Xmp.com}. Details should await a different time.

\bigskip
\noindent
{\bf Acknowledgements:} 
The paper is partially supported by INFN-Trieste. 
GL is supported by INFN, Iniziativa Specifica GAST,
by INDAM - GNSAGA and by the INDAM-CNRS IRL-LYSM. 
JA is supported by grant 2017-03710 from the
Swedish Research Council.  
 Furthermore, JA would like to thank the Department of Mathematics 
 and Geosciences, University of Trieste for hospitality.

\bibliographystyle{alpha}

\end{document}